\documentclass[12pt, reqno]{amsart}
\usepackage{ amsmath,amsthm, amscd, amsfonts, amssymb, graphicx, color}
\usepackage[bookmarksnumbered, colorlinks, plainpages]{hyperref}
\newtheorem{thm}{Theorem}[section]
\newtheorem{cor}[thm]{Corollary}
\newtheorem{lem}[thm]{Lemma}

\newtheorem{exam}[thm]{Example}
\numberwithin{equation}{section}

\begin{document}

\title{Expressions for the g-Drazin inverse in a Banach algebra}

\author{Huanyin Chen}
\author{Marjan Sheibani}
\address{
Department of Mathematics\\ Hangzhou Normal University\\ Hang -zhou, China}
\email{<huanyinchen@aliyun.com>}
\address{Women's University of Semnan (Farzanegan), Semnan, Iran}
\email{<sheibani@fgusem.ac.ir>}


\begin{abstract}
We explore the generalized Drazin inverse in a Banach algebra. Let $\mathcal{A}$ be a Banach algebra, and let $a,b\in \mathcal{A}^{d}$.
If $ab=\lambda a^{\pi}bab^{\pi}$ then $a+b\in \mathcal{A}^{d}$. The explicit representation of $(a+b)^d$ is also presented. As applications of our results, we present new representations for the generalized Drazin inverse of a block matrix in a Banach algebra. The main results of Liu and Qin [Representations for the generalized Drazin inverse of the sum in a Banach algebra and its application for some operator matrices, Sci. World J., {\bf 2015}, 156934.8] are extended.
\end{abstract}

\maketitle

\section{Introduction}

Throughout the paper, $\mathcal{A}$ is a complex Banach algebra with an identity. The commutant of $a\in \mathcal{A}$ is defined by $comm(a)=\{x\in
\mathcal{A}~|~xa=ax\}$. An element $a$ in $\mathcal{A}$ has g-Drazin inverse (that is, generalized Drazin inverse) provided that there exists $b\in comm(a)$ such that $b=bab$ and $a-a^2b\in \mathcal{A}^{qnil}.$ Here, $\mathcal{A}^{qnil}$ is the set of all quasinilpotents in $\mathcal{A}$, i.e.,
$$\mathcal{A}^{qnil}=\{a\in \mathcal{A}~|~1+ax\in U(\mathcal{A})~\mbox{for
every}~x\in comm(a)\}.$$  For a Banach algebra $\mathcal{A}$ we have $$a\in \mathcal{A}^{qnil}\Leftrightarrow
\lim\limits_{n\to\infty}\parallel a^n\parallel^{\frac{1}{n}}=0\Leftrightarrow 1+\lambda a\in U(\mathcal{A})~\mbox{for any}~ \lambda\in {\Bbb C}.$$
We use $\mathcal{A}^{d}$ to denote the set of all g-Drazin invertible elements in $\mathcal{A}$.
As it is well known, $a\in \mathcal{A}^{d}$ if and only if there exists an idempotent $p\in comm(a)$ such that $a+p$ is invertible and $ap\in \mathcal{A}^{qnil}$ (see ~\cite[Theorem 4.2]{J}). The objective of this paper is to further explore the generalized Drazin inverse in a Banach algebra.

The g-Drazin invertibility of the sum of two elements in a Banach algebra is attractive. Many authors have studied such problems from many different views, e.g., ~\cite{CD, CL, DC, DW, K, M, MZ2, Z3}. In Section 2, we investigate when the sum of two g-Drazin invertible elements in a Banach algebra has g-Drazin inverse. Let $\mathcal{A}$ be a Banach algebra, and let $a,b\in \mathcal{A}^{d}$. If $ab=\lambda a^{\pi}bab^{\pi}$, we prove that $a+b\in \mathcal{A}^{d}$. The explicit representation of $(a+b)^d$ is also presented. This extends ~\cite[Theorem 4]{K} to more general setting.

It is a hard problem to find a formula for the g-Drazin inverse of a block matrix. There have been many papers on this subject under different conditions, e.g., ~\cite{DC, D, DM, DD, ZM}. Let $M=\left(
  \begin{array}{cc}
    A & B \\
    C & D
  \end{array}
\right)\in M_2(\mathcal{A})$, $A$ and $D$ have g-Drazin inverses. If $a\in \mathcal{A}$ has g-Drazin inverse $a^{d}$. The element $a^{\pi}=1-aa^{d}$ is called the spectral idempotent of $a$. In Section 3, we are concerned with new conditions on spectral idempotent matrices under which $M$ has g-Drazin inverse. If $BD=\lambda (BC)^{\pi}ABD^{\pi}$ and $CA=\lambda (CB)^{\pi}DCA^{\pi}$, we prove that $M\in M_2(\mathcal{A})^d$. The formula for $M^d$ is given as well. This extends ~\cite[Theorem 10]{K} to the wider case.

Finally, in the last section, we present certain simpler representations of the g-Drazin inverse of the block matrix $M$. If $BD=\lambda A^{\pi}AB, DC=\lambda^{-1}D^{\pi}CAA^{\pi}$ and $BC=0$, then $M\in M_2(\mathcal{A})^d$ and
$$M^d=\left(
\begin{array}{cc}
A^d&(A^d)^2B+\sum\limits_{n=0}^{\infty}A^nB(D^d)^{n+2}\\
C(A^d)^2&D^d+C(A^d)^3B+\sum\limits_{n=0}^{\infty}\sum\limits_{k=0}^{\infty}D^{k-1}CA^{n-k}B(D^d)^{n+2}
\end{array}
\right).$$

Let $p\in \mathcal{A}$ be an idempotent, and
let $x\in \mathcal{A}$. Then we write $$x=pxp+px(1-p)+(1-p)xp+(1-p)x(1-p),$$
and induce a Pierce representation given by the matrix
$$x=\left(\begin{array}{cc}
pxp&px(1-p)\\
(1-p)xp&(1-p)x(1-p)
\end{array}
\right)_p.$$ We use $\mathcal{A}^{-1}$ to denote the set of all invertible elements in $\mathcal{A}$. $\lambda$ always stands for a nonzero complex number.

\section{Additive results}

In this section we establish some additive properties of g-Drazin inverse in Banach algebras. We begin with
\begin{lem} Let $$x=\left(\begin{array}{cc}
a&0\\
c&b
\end{array}
\right)_p~\mbox{or}~\left(\begin{array}{cc}
b&c\\
0&a
\end{array}
\right)_p$$  Then $$x^d=\left(\begin{array}{cc}
a^d&0\\
z&b^d
\end{array}
\right)_p,~\mbox{or}~ \left(\begin{array}{cc}
b^d&z\\
0&a^d
\end{array}
\right)_p,$$ where $$\begin{array}{c}
z=(b^d)^2\big(\sum\limits_{i=0}^{\infty}(b^d)^ica^i\big)a^{\pi}+b^{\pi}\big(\sum\limits_{i=0}^{\infty}b^ic(a^d)^i\big)(a^d)^2-b^dca^d.
\end{array}$$
\end{lem}
\begin{proof} See ~\cite[Lemma 2.1]{CD}.\end{proof}

\begin{lem} Let $\mathcal{A}$ be a Banach algebra, and let $a,b\in\mathcal{A}^{qnil}$. If $ab=\lambda ba$, then $a+b\in \mathcal{A}^{qnil}$.
\end{lem}
\begin{proof} See~\cite[Lemma 2.1]{C} and \cite[Lemma 2.1]{DW2}.\end{proof}

\begin{lem} Let $\mathcal{A}$ be a Banach algebra, and let $a\in \mathcal{A}^{qnil},b\in \mathcal{A}^{d}$. If $$ab=\lambda bab^{\pi},$$ then $a+b\in \mathcal{A}^{d}$ and
$$(a+b)^d=b^d+\sum\limits_{n=0}^{\infty}(b^d)^{n+2}a(a+b)^n.$$\end{lem}
\begin{proof} Let $p=bb^{d}$. Then we have
$$b=\left(\begin{array}{cc}
b_1&0\\
0&b_2
\end{array}
\right)_p, a=\left(\begin{array}{cc}
a_{1}&a_{2}\\
a_{3}&a_4
\end{array}
\right)_p.$$ Hence,
$$b^d=\left(\begin{array}{cc}
b_1^{-1}&0\\
0&0
\end{array}
\right)_p~\mbox{and}~b^{\pi}=\left(\begin{array}{cc}
0&0\\
0&1-bb^d
\end{array}
\right)_p.$$ Since $ab=\lambda bab^{\pi},$, we get
$$\left(\begin{array}{cc}
a_1b_1&a_2b_2\\
a_3b_1&a_4b_2
\end{array}
\right)_p=ab=\lambda bab^{\pi}=\left(\begin{array}{cc}
0&\lambda b_1a_2\\
0&\lambda b_2a_4
\end{array}
\right)_p.$$
Thus, $a_1b_1=0$ and $a_3b_1=0$, and then $a_1=0$ and $a_3=0$. Obviously, $b_2=b-b^2b^d\in ((1-p)\mathcal{A}(1-p))^{qnil}$.
Since $ab=\lambda bab^{\pi},$ we have $abb^d=\lambda bab^{\pi}b^d=0$. Hence $a(1-bb^d)=a\in \mathcal{A}^{qnil}$. In view of Cline's formula (see~\cite[Theorem 2.1]{L}),
we prove that $a_4=b^{\pi}ab^{\pi}\in \mathcal{A}^{qnil}$. As $a_4b_2=\lambda b_2a_4$, by Lemma 2.2, we show that
$a_4+b_2\in ((1-p)\mathcal{A}(1-p))^{qnil}$, i.e., $(a_4+b_2)^d=0$.

Since $$a+b=\left(\begin{array}{cc}
b_1&a_2\\
0&a_4+b_2
\end{array}
\right)_p,$$ it follows by Lemma 2.1 that $$(a+b)^d=\left(\begin{array}{cc}
b_1&a_2\\
0&a_4+b_2
\end{array}
\right)^d=\left(\begin{array}{cc}
b_1^{-1}&z\\
0&0
\end{array}
\right)_p,$$ where $z=(b^d)^2\big(\sum\limits_{i=0}^{\infty}(b^d)^ia(a_4+b_2)^i\big).$ Since $abb^d=0$, we derive  $$(a+b)^d=b^d+\sum\limits_{n=0}^{\infty}(b^d)^{n+2}a(a+b)^n.$$\end{proof}

Now we state one of our main results.

\begin{thm} Let $\mathcal{A}$ be a Banach algebra, and let $a,b\in \mathcal{A}^{d}$. If $$ab=\lambda a^{\pi}bab^{\pi},$$ then $a+b\in \mathcal{A}^{d}$ and
$$\begin{array}{lll}
(a+b)^d&=&b^{\pi}a^d+b^da^{\pi}+\sum\limits_{n=0}^{\infty}(b^d)^{n+2}a(a+b)^na^{\pi}\\
&+&b^{\pi}\sum\limits_{n=0}^{\infty}(a+b)^nb(a^d)^{n+2}\\
&-&\sum\limits_{n=0}^{\infty}\sum\limits_{k=0}^{\infty}(b^d)^{k+1}a(a+b)^{n+k}b(a^d)^{n+2}\\
&-&\sum\limits_{n=0}^{\infty}(b^d)^{n+2}a(a+b)^nba^d.
\end{array}$$\end{thm}
\begin{proof} Let $p=aa^{d}$. Then we have
$$a=\left(\begin{array}{cc}
a_1&0\\
0&a_2
\end{array}
\right)_p, b=\left(\begin{array}{cc}
b_{11}&b_{12}\\
b_{1}&b_2
\end{array}
\right)_p.$$ Since $ab=\lambda a^{\pi}bab^{\pi},$ we have $aa^db=\lambda a^da^{\pi}bab^{\pi}=0$; hence,
$b_{11}=b_{12}=0$. Thus,
$$a=\left(\begin{array}{cc}
a_1&0\\
0&a_2
\end{array}
\right)_p, b=\left(\begin{array}{cc}
0&0\\
b_{1}&b_2
\end{array}
\right)_p,$$ Hence,
$$a^d=\left(\begin{array}{cc}
a_1^d&0\\
0&0
\end{array}
\right)_p, b^d=\left(\begin{array}{cc}
0&0\\
(b_2^d)^2b_{1}&b_2^d
\end{array}
\right)_p.$$ Thus, we have
$$a^{\pi}=\left(\begin{array}{cc}
0&0\\
0&1-aa^d
\end{array}
\right)_p, b^{\pi}=\left(\begin{array}{cc}
1&0\\
-b_2^db_{1}&b_2^{\pi}
\end{array}
\right)_p.$$
Clearly, $a_2=(1-p)a(1-p)=a-a^2a^{d}\in \mathcal{A}^{qnil}$.
Since $(1-aa^d)b=b\in \mathcal{A}^{qnil}$, it follows by Cline's formula that $b_2=a^{\pi}ba^{\pi}\in ((1-p)\mathcal{A}(1-p))^{d}$.
As $ab=\lambda a^{\pi}bab^{\pi},$ we have
$$\begin{array}{ll}
\left(
\begin{array}{cc}
0&0\\
a_2b_1&a_2b_2
\end{array}
\right)&=ab=\lambda a^{\pi}bab^{\pi}\\
&=\lambda \left(
\begin{array}{cc}
0&0\\
b_1a_1-b_2^db_1&b_2a_2b_2^{\pi}
\end{array}
\right)\\
\end{array},$$ and then $$a_2b_2=\lambda b_2a_2b_2^{\pi}.$$ In view of Lemma 2.3,
$$(a_2+b_2)^d=b_2^d+\sum\limits_{n=0}^{\infty}(b_2^d)^{n+2}a_2(a_2+b_2)^n.$$
By virtue of Lemma 2.1, we have $$(a+b)^d=\left(\begin{array}{cc}
a_1^d&0\\
z&(a_2+b_2)^d
\end{array}
\right)=\left(\begin{array}{cc}
a^d&0\\
z&(a_2+b_2)^d\end{array}
\right),$$ where $$z=(a_2+b_2)^{\pi}\big(\sum\limits_{i=0}^{\infty}(a_2+b_2)^ib(a^d)^i\big)(a^d)^2-(a_2+b_2)^dba^d.$$
We easily see that $a_2b_2^d=(\lambda b_2a_2b_2^{\pi})(b_2^d)^2=0$; hence,
$$\begin{array}{lll}
(a_2+b_2)^{\pi}&=&(1-aa^d)-b_2b_2^d-\sum\limits_{n=0}^{\infty}(b_2^d)^{n+1}a_2(a_2+b_2)^n\\
&=&b_2^{\pi}-\sum\limits_{n=0}^{\infty}(b_2^d)^{n+1}a_2(a_2+b_2)^n.
\end{array}$$
Moreover, we have $$\begin{array}{lll}
z&=&\sum\limits_{i=0}^{\infty}b_2^{\pi}(a_2+b_2)^ib(a^d)^{i+2}\\
&-&\sum\limits_{n=0}^{\infty}\sum\limits_{i=0}^{\infty}(b_2^d)^{n+1}a_2(a_2+b_2)^{n+i}b(a^d)^{i+2}\\
&-&b_2^dba^d+\sum\limits_{n=0}^{\infty}(b_2^d)^{n+2}a_2(a_2+b_2)^nba^d.
\end{array}$$
Clearly, $a^db=ab^d=0$, we easily check that
$$\begin{array}{c}
\left(\begin{array}{cc}
a^d&0\\
-b_2^db_1a^d&0
\end{array}
\right)=b^{\pi}a^d,\\
\left(\begin{array}{cc}
0&0\\
b_2^{\pi}(a_2+b_2)^ib(a^d)^{i+2}&0
\end{array}
\right)=b^{\pi}(a+b)^ib(a^d)^{i+2},\\
\left(\begin{array}{cc}
0&0\\
(b_2^d)^{n+2}a_2(a_2+b_2)^nba^d&0
\end{array}
\right)=(b^d)^{n+2}a(a+b)^nba^d,\\
\left(\begin{array}{cc}
0&0\\
(b_2^d)^{n+1}a_2(a_2+b_2)^{n+i}b(a^d)^{i+2}&0
\end{array}
\right)=(b^d)^{n+1}a(a+b)^{n+i}b(a^d)^{i+2}.
\end{array}$$
Moreover, we have $$
\begin{array}{c}
\left(\begin{array}{cc}
0&0\\
0&b_4^d
\end{array}
\right)=b^d(1-aa^d),\\
\left(\begin{array}{cc}
0&0\\
0&(b_4^d)^{n+2}a_2(a_2+b_4)^n
\end{array}
\right)=(b^d)^{n+2}a(a+b)^n.
\end{array}$$
Therefore we easily obtain the result.\end{proof}

\begin{exam} Let $\mathcal{A}=M_3({\Bbb C})$ and let $$a=\left(\begin{array}{ccc}
0&0&0\\
1&0&0\\
0&1&0
\end{array}
\right), b=\left(\begin{array}{ccc}
0&0&0\\
1&0&0\\
0&2&0
\end{array}
\right)\in \mathcal{A}^d.$$ Then $ab=\frac{1}{2}a^{\pi}bab^{\pi},$ while $ab\neq a^{\pi}bab^{\pi}$.
\end{exam}
\begin{proof} It is clear that $a^3=b^3=0$, then $a^d=b^d=0$ which implies that $a^{\pi}=b^{\pi}=I_3.$
$$ab= \left(\begin{array}{ccc}
0&0&0\\
0&0&0\\
1&0&0
\end{array}
\right)=\frac{1}{2}a^{\pi}bab^{\pi},$$ while $ a^{\pi}bab^{\pi} = \left(\begin{array}{ccc}
0&0&0\\
0&0&0\\
1&0&0
\end{array}
\right)\neq ab$.  \end{proof}

\section{Block operator matrices}

In this section, we we turn to study the g-Drazin inverse of the block matrix $M$ by applying Theorem 2.4. We now derive

\begin{thm} Let $M=\left(
  \begin{array}{cc}
    A & B \\
    C & D
  \end{array}
\right)\in M_2(\mathcal{A})$, $A$ and $D$ have g-Drazin inverses. If $BD=\lambda (BC)^{\pi}ABD^{\pi}$ and $CA=\lambda (CB)^{\pi}DCA^{\pi}$, then $M\in M_2(\mathcal{A})^d$ and
$$\begin{array}{lll}
M^d&=&\left(
\begin{array}{cc}
A^d(BC)^{\pi}&A^{\pi}B(CB)^d\\
D^{\pi}C(BC)^d&D^d(CB)^{\pi}
\end{array}
\right)+\sum\limits_{n=0}^{\infty}(P^d)^{n+2}QM^nQ^{\pi}\\
&+&P^{\pi}\sum\limits_{n=0}^{\infty}M^nP(Q^d)^{n+2}\\
&-&\sum\limits_{n=0}^{\infty}\sum\limits_{k=0}^{\infty}(P^d)^{k+1}QM^{n+k}P(Q^d)^{n+2}\\
&-&\sum\limits_{n=0}^{\infty}(P^d)^{n+2}QM^nPQ^d.
\end{array}$$\end{thm}
\begin{proof} Clearly, we have $M=P+Q$, where
$$P=\left(
\begin{array}{cc}
A&0\\
0&D
\end{array}
\right), Q=\left(
\begin{array}{cc}
0&B\\
C&0
\end{array}
\right).$$ Then we have
$$\begin{array}{c}
P^d=\left(
\begin{array}{cc}
A^d&0\\
0&D^d
\end{array}
\right),P^{\pi}=\left(
\begin{array}{cc}
A^{\pi}&0\\
0&D^{\pi}
\end{array}
\right);\\
Q^2=\left(
\begin{array}{cc}
BC&0\\
0&CB
\end{array}
\right), (Q^2)^d=\left(
\begin{array}{cc}
(BC)^d&0\\
0&(CB)^d
\end{array}
\right).
\end{array}$$ By using Cline's formula, we get $$Q^d=Q(Q^2)^d=\left(
\begin{array}{cc}
0&B(CB)^d\\
C(BC)^d&0
\end{array}
\right).$$ Hence,
$$Q^{\pi}=\left(
\begin{array}{cc}
(BC)^{\pi}&0\\
0&(CB)^{\pi}
\end{array}
\right).$$
Clearly,
$$\begin{array}{c}
PQ=\left(
\begin{array}{cc}
0&AB\\
DC&0
\end{array}
\right),\\
QP=\left(
\begin{array}{cc}
0&BD\\
CA&0
\end{array}
\right),
\end{array}$$
and so
$$Q^{\pi}PQP^{\pi}=\left(
  \begin{array}{cc}
    0&(BC)^{\pi}ABD^{\pi}\\
    (CB)^{\pi}DCA^{\pi}& D
  \end{array}
\right).$$
By hypothesis, we have
$$QP=\lambda Q^{\pi}PQP^{\pi}.$$¡¡According to Theorem 2.4,
$M$ has g-Drain invesse. The representation of $M^d$ is easily obtained by Theorem 2.4.\end{proof}

\begin{cor} Let $M=\left(
  \begin{array}{cc}
    A & B \\
    C & D
  \end{array}
\right)\in M_2(\mathcal{A})$, $A$ and $D$ have g-Drazin inverses. If $BD=\lambda ABD^{\pi}, CA=\lambda DCA^{\pi}$ and $BC=0$, then $M\in M_2(\mathcal{A})^d$ and
$$\begin{array}{lll}
M^d&=&\left(
\begin{array}{cc}
A^d&0\\
0&D^d
\end{array}
\right)+\sum\limits_{n=0}^{\infty}(P^d)^{n+2}QM^n.
\end{array}$$\end{cor}\begin{proof} Since $BC=0$, we see that $(BC)^{\pi}=I=(CB)^{\pi}$. Moreover, we see that $$(BC)^d=0, B(CB)^d=B(CB)((BC)^d)^2=0, Q^d=0, Q^{\pi}=I.$$ This completes the proof by Theorem 3.1.\end{proof}

In a similar way as it was dong in Theorem 3.1, using the another splitting, we have

\begin{thm} Let $M=\left(
  \begin{array}{cc}
    A & B \\
    C & D
  \end{array}
\right)\in M_2(\mathcal{A})$, $A$ and $D$ have g-Drazin inverses. If $AB=\lambda A^{\pi}BD(CB)^{\pi}$ and $DC=\lambda D^{\pi}CA(BC)^2$, then $M\in M_2(\mathcal{A})^d$ and
$$\begin{array}{lll}
M^d&=&\left(
\begin{array}{cc}
(BC)^{\pi}A^d&B(CB)^dD^{\pi}\\
C(BC)^dA^{\pi}&(CB)^{\pi}D^d
\end{array}
\right)+\sum\limits_{n=0}^{\infty}(Q^d)^{n+2}PM^nP^{\pi}\\
&+&Q^{\pi}\sum\limits_{n=0}^{\infty}M^nQ(P^d)^{n+2}\\
&-&\sum\limits_{n=0}^{\infty}\sum\limits_{k=0}^{\infty}(Q^d)^{k+1}PM^{n+k}Q(P^d)^{n+2}\\
&-&\sum\limits_{n=0}^{\infty}(Q^d)^{n+2}PM^nQP^d.
\end{array}$$\end{thm}
\begin{proof} Construct $P$ and $Q$ as in Theorem 3.1, we have
$$\begin{array}{c}
PQ=\left(
\begin{array}{cc}
0&AB\\
DC&0
\end{array}
\right),\\
P^{\pi}QPQ^{\pi}=\left(
\begin{array}{cc}
0&A^{\pi}BD(CB)^{\pi}\\
D^{\pi}CA(BC)^{\pi}&0
\end{array}
\right).
\end{array}$$
By hypothesis, we see that $PQ=\lambda P^{\pi}QPQ^{\pi}$. This completes the proof by Theorem 2.4.\end{proof}

As a consequence of the above, we derive

\begin{cor} Let $M=\left(
  \begin{array}{cc}
    A & B \\
    C & D
  \end{array}
\right)\in M_2(\mathcal{A})$, $A$ and $D$ have g-Drazin inverses. If $AB=\lambda A^{\pi}BD, DC=0$ and $BC=0$, then $M\in M_2(\mathcal{A})^d$ and
$$\begin{array}{lll}
M^d&=&\left(
\begin{array}{cc}
A^d&0\\
0&D^d
\end{array}
\right)+\sum\limits_{n=0}^{\infty}M^nQ(P^d)^{n+2}.
\end{array}$$\end{cor}

\section{Certain simpler expressions}

Let $M=\left(
  \begin{array}{cc}
    A & B \\
    C & D
  \end{array}
\right)\in M_2(\mathcal{A})$. The aim of this section is to present certain simpler representations of the g-Drazin inverse of the block matrix $M$ in the case
$BC=0$ or $CB=0$. We now come to the main result of this section.

\begin{thm} Let $A$ and $D$ have g-Drazin inverses. If $BD=\lambda A^{\pi}AB, DC=\lambda^{-1}D^{\pi}CAA^{\pi}$ and $BC=0$, then $M\in M_2(\mathcal{A})^d$ and
$$M^d=\left(
\begin{array}{cc}
A^d&(A^d)^2B+\sum\limits_{n=0}^{\infty}A^nB(D^d)^{n+2}\\
C(A^d)^2&D^d+C(A^d)^3B+\sum\limits_{n=0}^{\infty}\sum\limits_{k=0}^{\infty}D^{k-1}CA^{n-k}B(D^d)^{n+2}
\end{array}
\right).$$\end{thm}
\begin{proof} Write $M=P+Q$, where $$P=\left(
  \begin{array}{cc}
    AA^{\pi} & 0 \\
    0 & D
  \end{array}
\right), Q=\left(
  \begin{array}{cc}
    A^2A^d & B \\
    C & 0
  \end{array}
\right).$$ Then $$P^d=\left(
  \begin{array}{cc}
    0 & 0 \\
    0 & D^d
  \end{array}
\right), P^{\pi}=\left(
  \begin{array}{cc}
   I & 0 \\
    0 & D^{\pi}
  \end{array}
\right).$$ Since $BC=0$, we have
$$Q^d=\left(
  \begin{array}{cc}
    A^d & (A^d)^2B \\
    C(A^d)^2 & C(A^d)^3B
  \end{array}
\right), Q^{\pi}=\left(
  \begin{array}{cc}
    A^{\pi} & -A^dB \\
    -CA^d & I-C(A^d)^2B
  \end{array}
\right).$$ Since $BD=\lambda A^{\pi}AB$ and $BC=0$, we see that $BDC=(\lambda A^{\pi}AB)C=0$. As $DC=\lambda^{-1} D^{\pi}CAA^{\pi}$, we have
$$PQ=\left(
  \begin{array}{cc}
   0 & A^{\pi}AB \\
    DC & 0
  \end{array}
\right)=\lambda^{-1}\left(
  \begin{array}{cc}
  0 & BD \\
    D^{\pi}CAA^{\pi} & 0
  \end{array}
\right)=\lambda^{-1}P^{\pi}QPQ^{\pi}.$$ In view of Theorem 2.4, we have
$$\begin{array}{lll}
M^d&=&Q^{\pi}P^d+Q^dP^{\pi}+\sum\limits_{n=0}^{\infty}(Q^d)^{n+2}PM^nP^{\pi}\\
&+&Q^{\pi}\sum\limits_{n=0}^{\infty}M^nQ(P^d)^{n+2}\\
&-&\sum\limits_{n=0}^{\infty}\sum\limits_{k=0}^{\infty}(Q^d)^{k+1}PM^{n+k}Q(P^d)^{n+2}\\
&-&\sum\limits_{n=0}^{\infty}(Q^d)^{n+2}PM^nQP^d.
\end{array}$$
Since $BD=\lambda A^{\pi}AB$, we see that $A^dBD=0$, and then $$Q^dP=\left(
  \begin{array}{cc}
    A^d & (A^d)^2B \\
    C(A^d)^2 & C(A^d)^3B
  \end{array}
\right)\left(
  \begin{array}{cc}
    AA^{\pi} & 0 \\
    0 & D
  \end{array}
\right)=0.$$ Therefore $$M^d=P^d+Q^d+Q^{\pi}\sum\limits_{n=0}^{\infty}M^nQ(P^d)^{n+2}.$$ Moreover, we have $BD^nC=0$ for any $n\in {\Bbb N}$. Therefore
$$Q^{\pi}\sum\limits_{n=0}^{\infty}M^nQ(P^d)^{n+2}=\left(
  \begin{array}{cc}
   0 & \sum\limits_{n=1}^{\infty}A^nB(D^d)^{n+2} \\
   0 & \sum\limits_{n=1}^{\infty}\sum\limits_{i=1}^nD^{i-1}CA^{n-i}BD^{n+2}
  \end{array}
\right),$$ as desired.\end{proof}

\begin{cor} Let $A$ and $D$ have g-Drazin inverses. If $CA=\lambda D^{\pi}DC, AB=\lambda A^{\pi}BDD^{\pi}$ and $CB=0$, then $M\in M_2(\mathcal{A})^d$ and
$$M^d=\left(
\begin{array}{cc}
A^d+B(D^d)^3C+\sum\limits_{n=0}^{\infty}\sum\limits_{k=0}^{\infty}A^{k-1}BD^{n-k}C(A^d)^{n+2}&B(D^d)^2\\
(D^d)^2C+\sum\limits_{n=0}^{\infty}D^nC(A^d)^{n+2}&D^d
\end{array}
\right).$$\end{cor}
\begin{proof} Obviously, $$\left(
\begin{array}{cc}
A&B\\
C&D
\end{array}
\right)=\left(
\begin{array}{cc}
0&I\\
I&0
\end{array}
\right)\left(
\begin{array}{cc}
D&C\\
B&A
\end{array}
\right)\left(
\begin{array}{cc}
0&I\\
I&0
\end{array}
\right).$$ Applying Theorem 4.1 to $\left(
\begin{array}{cc}
D&C\\
B&A
\end{array}
\right)$, we see that it has g-Drazin inverse and $$\begin{array}{ll}
&\left(
\begin{array}{cc}
D&C\\
B&A
\end{array}
\right)^d\\
=&\left(\begin{array}{cc}
D^d&(D^d)^2C+\sum\limits_{n=0}^{\infty}D^nC(A^d)^{n+2}\\
B(D^d)^2&A^d+B(D^d)^3C+\sum\limits_{n=0}^{\infty}\sum\limits_{k=0}^{\infty}A^{k-1}BD^{n-k}C(A^d)^{n+2}
\end{array}
\right).
\end{array}$$  Therefore $$\begin{array}{lll}
M^d&=&\left(
\begin{array}{cc}
0&I\\
I&0
\end{array}
\right)\left(
\begin{array}{cc}
D&C\\
B&A
\end{array}
\right)^d\left(
\begin{array}{cc}
0&I\\
I&0
\end{array}
\right),
\end{array}$$ as desired.\end{proof}

Now we are ready to prove the other main theorem in this section.

\begin{thm} Let $A$ and $D$ have g-Drazin inverses. If $AB=\lambda A^{\pi}BD, DC=\lambda D^{\pi}CA$ and $BC=0$, then $M\in M_2(\mathcal{A})^d$ and
$$M^d=\left(
\begin{array}{cc}
A^d&0\\
0&D^d
\end{array}
\right)+\sum\limits_{n=0}^{\infty}M^n\left(
\begin{array}{cc}
0&B(D^d)^{n+2}\\
C(A^d)^{n+2}&0
\end{array}
\right).$$\end{thm}
\begin{proof} Write $M=P+Q$, where $$P=\left(
  \begin{array}{cc}
    A & 0 \\
    0 & D
  \end{array}
\right), Q=\left(
  \begin{array}{cc}
   0 & B \\
    C & 0
  \end{array}
\right).$$ Then $$P^d=\left(
  \begin{array}{cc}
    A^d & 0 \\
    0 & D^d
  \end{array}
\right), P^{\pi}=\left(
  \begin{array}{cc}
   A^{\pi} & 0 \\
    0 & D^{\pi}
  \end{array}
\right).$$ As $BC=0$, we see that $Q^3=0$, and so $Q^d=0, Q^{\pi}=I$. We easily check that
$$PQ=\left(
  \begin{array}{cc}
   0 & AB \\
    DC & 0
  \end{array}
\right)=\lambda\left(
  \begin{array}{cc}
  0 & A^{\pi}BD \\
    D^{\pi}CA & 0
  \end{array}
\right)=\lambda P^{\pi}QPQ^{\pi}.$$ Since $Q^d=0$, it follows by Theorem 2.4 that
$$M^d=P^d+\sum\limits_{n=0}^{\infty}M^nQ(P^d)^{n+2}.$$ Moreover, we have
$$\sum\limits_{n=0}^{\infty}M^nQ(P^d)^{n+2}=\sum\limits_{n=1}^{\infty}M^n\left(
  \begin{array}{cc}
   0 & B(D^d)^{n+2} \\
   C(A^d)^{n+2}&0
  \end{array}
\right),$$ as required.\end{proof}

\begin{exam} Let $M=\left(
  \begin{array}{cc}
    A & B \\
    C & D
  \end{array}
\right)\in M_8({\Bbb C})$, where $$A=D=\left(
  \begin{array}{cccc}
   0&1&0&0\\
   0&0&1&0\\
   0&0&0&1\\
   0&0&0&0
  \end{array}
\right), B=C=\left(
  \begin{array}{cccc}
   0&0&1&0\\
   0&0&0&3\\
   0&0&0&0\\
   0&0&0&0
  \end{array}
\right)\in M_4({\Bbb C}).$$ Then $$AB=3A^{\pi}BD, DC=3D^{\pi}CA~\mbox{and}~BC=0.$$\end{exam}
\begin{proof} As $A, B, C, D$ are nilpotent, so $A^{\pi}=B^{\pi}=C^{\pi}=D^{\pi}=I_4$. It is clear by computing that $$AB=3A^{\pi}BD, DC=3D^{\pi}CA~\mbox{and}~BC=0.$$ In this case, $AB\neq A^{\pi}BD.$\end{proof}

\end{document}